\documentclass[12pt]{article}
\usepackage{amssymb,latexsym,amscd,amsmath,amsfonts,amsthm,enumerate}

\newtheorem{theorem}{Theorem}[section]
\newtheorem{lemma}[theorem]{Lemma}
\newtheorem{proposition}[theorem]{Proposition}
\newtheorem{corollary}[theorem]{Corollary}

\newtheoremstyle{definition}% name
  {6pt}%      Space above
  {6pt}%      Space below
  {}%         Body font
  {}%         Indent amount (empty = no indent, \parindent = para indent)
  {\bfseries}% Thm head font
  {.}%        Punctuation after thm head
  {.5em}%     Space after thm head: " " = normal interword space;
  {}%  

\theoremstyle{definition}
\newtheorem{definition}[theorem]{Definition}

\newtheoremstyle{remark}% name
  {6pt}%      Space above
  {6pt}%      Space below
  {}%         Body font
  {}%         Indent amount (empty = no indent, \parindent = para indent)
  {\bfseries}% Thm head font
  {.}%        Punctuation after thm head
  {.5em}%     Space after thm head: " " = normal interword space;
  {}%  

\theoremstyle{remark}
\newtheorem{remark}[theorem]{Remark}
\newtheorem*{thr}{Main theorem}
\newtheorem*{de}{Definition}

\makeatletter
\renewcommand\@makefntext[1]{%
\setlength\parindent{1em}%
\noindent
\makebox[1.8em][r]{}{#1}}
\makeatother

\begin{document}
\title{\bf \large A NOTE ON JOINT REDUCTIONS \\ AND MIXED MULTIPLICITIES}  
\author{
 \centerline{Duong Quoc Viet and Le Van Dinh}\\
  \centerline{Truong Thi Hong Thanh}\\ 
\small Department of Mathematics, Hanoi National University of Education\\
\small 136 Xuan Thuy street, Hanoi, Vietnam\\
\small  duongquocviet@fmail.vnn.vn; dinhlevands@gmail.com;
 thanhtth@hnue.edu.vn\\
}

   \date{}
\maketitle
\centerline{\parbox[c]{11 cm}{
\small  ABSTRACT: Let  $(A, \frak m)$  be  a  noetherian   local ring with maximal ideal $\frak{m}$ and infinite residue field $k = A/\frak{m}.$  Let $J$ be an $\frak m$-primary ideal, $I_1,\ldots,I_s$ ideals of $A$, and $M$ a finitely generated  $A$-module. In this paper, we interpret  mixed multiplicities of $(I_1,\ldots, I_s,J)$ with respect to $M$ as multiplicities of joint reductions of them. This generalizes the Rees's theorem on mixed multiplicity \cite[Theorem 2.4]{Re}. As an application we show that mixed multiplicities are also  multiplicities of Rees's superficial sequences.}}

\section{Introduction}
Let  $(A, \frak m)$  be  a  noetherian   local ring with maximal ideal $\frak{m}$ and infinite residue field $k = A/\frak{m}.$  Let $M$ be a finitely generated  $A$-module, and
$I_1,\ldots,I_s$ ideals of $A$ such that $I= I_1\cdots I_s$  is not contained in $ \sqrt{\mathrm{Ann}_A{M}}$. Set $\dim M/0_M:I^\infty = q.$
Let $J$ be an $\frak m$-primary ideal.  By  \cite[Proposition 3.1]{Vi} (see also \cite {MV}),    
$$\ell_A\Big(\frac{J^{n_0}I_1^{n_1}\cdots I_s^{n_s}M}{J^{n_0 + 1}I_1^{n_1}\cdots I_s^{n_s}M}\Big)$$
is a polynomial of total degree $q-1$ for all large $n_0, n_1, \ldots, n_s.$  Write the terms of total degree $q-1$ in this polynomial in the form 
\footnotetext{\begin{itemize}
\item[ ] This research was in part supported by a grant from  NAFOSTED.
\item[ ]{\it Mathematics  Subject  Classification}(2010): Primary 13H15. Secondary 14C17, 13D40, 13C15. 
\item[ ]{\it Key words and phrases}: Mixed multiplicity,  multiplicity,  joint reduction, (FC)-sequence.
\end{itemize}}  
$$\sum_{k_0 + k_1 + \cdots + k_s = q - 1} e_A(J^{[k_0+1]},I_1^{[k_1]}, \ldots,I_s^{ [k_s]}, M)\frac{n_0^{k_0}n_1^{k_1}\cdots n_s^{k_s}}{k_0!k_1! \cdots k_s!},$$
then $e_A(J^{[k_0+1]},I_1^{[k_1]}, \ldots,I_s^{ [k_s]}, M)$ are non-negative integers not all zero. We call (see \cite {MV} and \cite{Ve})  $$e_A(J^{[k_0+1]},I_1^{[k_1]}, \ldots,I_s^{ [k_s]}, M)$$ the 
{\it mixed multiplicity of $(I_1,\ldots, I_s,J)$ with respect to $M$ of the type $( k_1, \ldots, k_s,k_0+1)$}.

Risler and Teissier in 1973 \cite{Te} defined  mixed multiplicities of $\frak m$-primary ideals and interpreted them as Hilbert-Samuel multiplicities of ideals generated by general elements. Katz and Verma in 1989 \cite{KV} started the investigation of mixed multiplicities of ideals of positive height. For the case of arbitrary ideals, the first author  in 2000 \cite{Vi} described  mixed multiplicities as  Hilbert-Samuel multiplicities via (FC)-sequences. Moreover, Trung and Verma in 2007 \cite{TV} interpreted mixed volumes of polytopes as  mixed multiplicities of ideals.     
In past years, the positivity and the
relationship between  mixed multiplicities  and  Hilbert-Samuel multiplicity of ideals have attracted much attention (see e.g. \cite{CP, DV, HHRT, KR1, KR2, MV,  Ro, Sw,Ve, Vi4, Vi2, Vi3, VM, VT}).
 
 We turn now to  Rees's work in 1984 \cite{Re}. The author of this work  built joint reductions
of $\frak m$-primary ideals and showed that each mixed multiplicity of $\frak m$-primary ideals is the multiplicity of a joint reduction of them.
 O'Carroll in 1987 \cite{Oc} proved the existence of joint reductions in the general case.
\begin{de}[$\mathrm{see\ Definition\ \ref{de31}}$]
 Let $\frak R$ be a subset of $\bigcup_{i=1}^sI_i$  
consisting of $k_1$ elements of $I_1,\ldots,k_s$ elements of  $I_s$
with $k_1,\ldots,k_s \ge 0.$  Set $\frak I_i = \frak R \bigcap I_i$ for $i=1,\ldots,s$ and $(\emptyset) = 0_A.$     Then 
$\frak R$ is called a {\it joint reduction} of $(I_1,\ldots,I_s)$ with respect to $M$ of the type $(k_1,\ldots,k_s)$ if $I_1^{n_1+1}\cdots I_s^{n_s+1}M = \sum_{i=1}^s(\frak I_i)
I_1^{n_1+1}\cdots I_i^{n_i}\cdots I_s^{n_s+1}M$ for all large integers $n_1,n_2,\ldots, n_s .$ 
\end{de}

Although the relationship between  mixed multiplicities  and  Hilbert-Samuel multiplicity of arbitrary ideals was solved, whether there is a similar result to Rees's theorem \cite[Theorem 2.4]{Re} for arbitrary ideals, i.e., whether there is a relationship between mixed multiplicities of arbitrary ideals and multiplicities of their joint reductions, is not yet known. And this problem  became an open question  in the theory of mixed multiplicities.
  The aim of this paper is to extend Rees's theorem to arbitrary ideals.         
 As one might expect, we obtain the following result.   

\begin{thr}[see Theorem \ref{th37}]
{\it  Let $M$ be a finitely generated $A$-module of Krull dimension $d > 0$. Let $J$ be an $\frak m$-primary ideal, and $I_1,\ldots, I_s$   ideals of $A.$ Set $I= I_1\cdots I_s$. Assume that $\mathrm{ht}\Big(\dfrac{I+\mathrm{Ann}_AM}{\mathrm{Ann}_AM}\Big) = h >0$ and $k_0, k_1,\ldots,k_s$ are non-negative integers such that $k_0+ k_1+\cdots+k_s = d-1$ and $ k_1+\cdots+k_s < h.$ 
Let $\frak R= \{x_1, \ldots, x_d\}$ be a joint reduction of 
 $(I_1,\ldots,I_s,J)$ with respect to $M$ of the type $(k_1,\ldots,k_s, k_0+1)$ such that     
$\frak R$ is a system of parameters for $M.$ Then $$e_A(J^{[k_0+1]}, I_1^{[k_1]},\ldots, I_s^{[k_s]}, M) = e_A(\frak R, M).$$}
\end{thr}

It should be noted that this theorem does not hold in general if one omits  the assumption $ k_1+\cdots+k_s < h$ (see Remark \ref{rm35}). Our approach, which is based on  the results in \cite {Vi} and \cite {VT1}, links the multiplicity of (FC)-sequences in \cite{Vi}  and the multiplicity of joint reductions via our studies on these sequences (see Proposition \ref{lm33} and 
Lemma \ref{lm36}). As an application of the main theorem, we  interpret mixed multiplicities as Hilbert-Samuel multiplicities of Rees's superficial sequences (see Remark \ref{rm20}) and recover Rees's theorem  in \cite[Theorem\ 2.4]{Re} (see Corollary \ref{co34}).

The paper is divided into three sections. Section 2 deals with the existences of Rees's superficial sequences,  weak-(FC)-sequences, and joint reductions.
 Apart from the proof of  the main theorem, Section 3 also contains a brief treatment on the relationship between mixed multiplicities and Rees's superficial sequences.

\section{(FC)-Sequences and Joint Reductions}

In general, the
relationship between  mixed multiplicities  and  Hilbert-Samuel multiplicity of arbitrary ideals was solved in \cite{Vi} by using (FC)-sequences. In this section, we give some results concerning Rees's superficial sequences,  (FC)-sequences, and joint reductions which will be used in this paper.

\begin{definition}[\cite{Vi}]  Let  $(A, \frak m)$  be  a  noetherian   local ring with maximal ideal $\frak{m}$ and infinite residue field $k = A/\frak{m}.$  Let $M$ be a finitely generated  $A$-module, and
$I_1,\ldots,I_s$ ideals of $A$ such that $I_1\cdots I_s$  is not contained in $ \sqrt{\mathrm{Ann}_A{M}}$. Set $I=I_1\cdots I_s.$  An element $x \in A$ is called an  (FC)-{\it element} of $(I_1,\ldots, I_s)$ with respect to $M$ if there exists $i \in \{ 1, \ldots, s\}$ such that $x \in I_i$ and 
the following conditions are satisfied:  
 \begin{enumerate}[(FC1):]
 \item $x{M}\bigcap {I_1}^{n_1} \cdots I_i^{n_i+1}\cdots I_s^{n_s}{M} 
= x{I_1}^{n_1}\cdots I_i^{n_i}\cdots I_s^{n_s}{M}$
for all $n_i\gg 0$ and all $n_1,\ldots,n_{i-1},n_{i+1}, \ldots, n_s \geq 0$. 

\item $x$ is an $I$-filter-regular element with respect to $M,$ i.e.,\;$0_M:x \subseteq 0_M: I^{\infty}.$
\item  $\dim M/(xM:I^\infty)=\dim M/(0_M:I^\infty)-1.$
 \end{enumerate}
       We call $x$ a {\it weak}-(FC)-{\it element} of $(I_1,\ldots, I_s)$ with respect to $M$ if $x$ satisfies the conditions (FC1) and (FC2). One can also call an element satisfying the condition (FC1) a {\it Rees's superficial element} of $(I_1,\ldots, I_s)$ with respect to $M$.

Let $x_1, \ldots, x_t$ be elements of $A$. For any $0\le i \le t,$ set\;      $M_i = \dfrac{M}{(x_1, \ldots, x_{i})M}$.  Then 
$x_1, \ldots, x_t$ is called an  (FC)-{\it sequence} (respectively, a {\it weak}-(FC)-{\it sequence}, a  {\it Rees's superficial sequence}) of $(I_1,\ldots, I_s)$ with respect to $M$ if $x_{i + 1}$ is an (FC)-element (respectively, a weak-(FC)-element, a Rees's superficial element) of $(I_1,\ldots, I_s)$ with respect to  $M_i$ for all $i = 0, \ldots, t - 1$. If an (FC)-sequence (respectively, a weak-(FC)-sequence, a  Rees's superficial sequence) of $(I_1,\ldots, I_s)$ with respect to $M$ consists of $k_1$ elements of $I_1,\ldots,k_s$ elements of  $I_s$
($k_1,\ldots,k_s \ge 0$) then it is called an (FC)-sequence (respectively, a weak-(FC)-sequence, a  Rees's superficial sequence) of $(I_1,\ldots,I_s)$ with respect to $M$ of the type $(k_1,\ldots,k_s)$. A  weak-(FC)-sequence $x_1, \ldots, x_t$ is called a {\it maximal weak-$(FC)$-sequence} if $I\nsubseteq \sqrt{\mathrm{Ann}_A{M_{t-1}}}$ and $I\subseteq \sqrt{\mathrm{Ann}_A{M_t}}.$  
\end{definition}

The existences of Rees's superficial elements and weak-(FC)-elements are established in Lemma \ref{lm32} and Proposition \ref{lm33}, which are the improvements of \cite[Lemma 1.2]{Re} and \cite[Remark 1]{Vi}, respectively.
 \enlargethispage{1cm}
\begin{lemma}[{see \cite [Lemma 1.2]{Re}}]\label{lm32}
 Let  $(A, \frak m)$  be  a  noetherian   local ring with maximal ideal $\frak{m}$ and infinite residue field $k = A/\frak{m},$ $I_1,\ldots,I_s$ ideals of $A,$ and $\Sigma$ a finite set of prime ideals of $A$.  Let $M$ be a finitely generated  $A$-module.  If for some $i\in\{1,\ldots,s\}$, $I_i$ is contained in no prime ideal of $\Sigma$, then there exists a non-empty Zariski open subset $U$ of $I_i/\frak mI_i$ such that for any element $x \in I_i$ with image $x + \frak mI_i \in U$, $x$ is not contained in any prime ideal of $\Sigma$ and $x$ is a Rees's superficial element of $(I_1,\ldots, I_s)$ with respect to $M,$ i.e.,  $$x{M}\bigcap {I_1}^{n_1} \cdots I_i^{n_i}\cdots I_s^{n_s}{M}= x{I_1}^{n_1}\cdots I_i^{n_i-1}\cdots I_s^{n_s}{M}$$ for all $n_i\gg 0$ and $n_1,\ldots,n_{i-1},n_{i+1}, \ldots, n_s \geq 0$. 
\end{lemma}

\begin{proof}
Set 
$\mathcal{R}= \bigoplus_{n_1, \ldots, n_s \in \;\;\mathbb{Z}} I_1^{n_1}\cdots I_s^{n_s}t_1^{n_1}\cdots t_s^{n_s}$  and 
$$ \mathcal{M} = \bigoplus_{n_1, \ldots, n_s \in \;\; \mathbb{Z}} I_1^{n_1}\cdots I_s^{n_s}t_1^{n_1}\cdots t_s^{n_s}M,$$
where $t_1, \ldots, t_s$ are indeterminates and $I_j^{n_j} = A$ for $n_j \leq 0$. Then $\mathcal{R}$ is a noetherian $\mathbb{Z}^s$-graded ring and $\mathcal{M}$ is a noetherian graded $\mathcal{R}$-module. 
Set $u_1= t_1^{-1}, \ldots, u_s = t_s^{-1}$. Since $u = u_1 \cdots u_s$ is a non-zerodivisor on $\mathcal{M}$,  the set of prime ideals  associated with $u^T\mathcal{M}$ is the same for all $T$ and so is finite (see \cite[Lemma 2.7]{Re}). 
Let $Q_1, \ldots, Q_t$ be all the associated prime ideals of $\mathcal{M}/u^T\mathcal{M}$ that do not contain $I_it_i$.
For each $l = 1, \ldots, t$, set $$Q'_l = \{a \in A \mid at_i \in Q_l\}.$$
Then $Q'_l$ is an ideal of $A$ that does not contain $I_i$. Let $W_l$ be the image of $Q'_l\bigcap I_i$ in $I_i/\frak mI_i$. 
Assume that $\Sigma = \{P_1, \ldots, P_r\}$. Also, for each $h = 1, \ldots, r$, let $V_h$ be the image of $P_h \bigcap I_i$ in $I_i/\frak mI_i$. By Nakayama's lemma, $W_l$ and $V_h$ are proper $k$-vector subspaces of $I_i/\frak mI_i$. Let $U$ be the complement of the union of all these $W_l$ and $V_h$. Since $k$ is an infinite field, $U$ is a non-empty Zariski open subset of $I_i/\frak mI_i$. 

Now, let $x \in I_i$ such that its image $x + \frak mI_i$ lies  in $U$. Then $x$ is not contained in any prime ideal in $\Sigma$ and $xt_i$ is not contained in  $ Q_l$ for all $l = 1, \ldots, t$. 
Using the same  arguments as in the proof of \cite[Lemma 1.2]{Re} we have 
$$x{M}\bigcap {I_1}^{n_1} \cdots I_i^{n_i}\cdots I_s^{n_s}{M}= x{I_1}^{n_1}\cdots I_i^{n_i-1}\cdots I_s^{n_s}{M}$$ for all $n_i\gg 0$ and $n_1,\ldots,n_{i-1},n_{i+1}, \ldots, n_s \geq 0$. 
 Hence $x$ is a Rees's superficial element of $(I_1,\ldots, I_s)$ with respect to $M$.
\end{proof}

\begin{proposition}[see {\cite[Remark 1]{Vi}}] \label{lm33}
 Let  $(A, \frak m)$  be  a  noetherian   local ring with maximal ideal $\frak{m}$ and infinite residue field $k = A/\frak{m},$ $I_1,\ldots,I_s$ ideals of $A.$ Let $M$ be a finitely generated  $A$-module. Assume that 
$I=I_1\cdots I_s \nsubseteq 
\sqrt{\mathrm{Ann}_A{M}}.$  
  Then for any $1 \le i \le s,$ there exists a non-empty Zariski open subset $U$ of $I_i/\frak mI_i$ such that if $x \in I_i$ with image $x + \frak mI_i \in U$ then $x$ is
a weak-$(FC)$-element of $(I_1,\ldots, I_s)$ with respect to $M.$ 
\end{proposition}

\begin{proof}
Set $\Sigma = \mathrm{Ass}_A\big(M/(0_{M}:I^\infty)\big).$  Since $I \nsubseteq 
\sqrt{\mathrm{Ann}_A{M}}$, $\Sigma \not= \emptyset$. It is easily seen that $\Sigma$ is a finite set and $\Sigma = \{P \in \mathrm{Ass}_AM \mid P \nsupseteq I\}$. 
By Lemma \ref{lm32}, for any $1 \le i \le s,$ there exists a non-empty Zariski open subset $U$ of $I_i/\frak mI_i$ such that if $x \in I_i$ with image $x + \frak mI_i \in U$ then $x$ is not contained in any prime ideal in $\Sigma$ and $x$ is a Rees's superficial element of $(I_1,\ldots, I_s)$ with respect to $M.$ 
Since $x \notin P$ for all $P \in \Sigma$, $x$ is also an $I$-filter-regular element with respect to $M$. Hence, for any element $x \in I_i$ with image $x + \frak mI_i \in U$,   $x$ is
a weak-(FC)-element of $(I_1,\ldots, I_s)$ with respect to $M.$ 
\end{proof}

Recall that the concept of joint reductions of $\mathfrak{m}$-primary ideals  was given by Rees in 1984  \cite{Re}. And he proved that mixed multiplicities of $\mathfrak{m}$-primary ideals are multiplicities of ideals generated by joint reductions.  
This concept was extended to the set of arbitrary ideals by \cite{MV, Oc,  Vi3}.

\begin{definition}\label{de31}
 Let $\frak R$ be a subset of $\bigcup_{i=1}^sI_i$  
consisting of $k_1$ elements of $I_1,\ldots,k_s$ elements of  $I_s$
with $k_1,\ldots,k_s \ge 0.$  Set $\frak I_i = \frak R \bigcap I_i$ for $i=1,\ldots,s$ and $(\emptyset) = 0_A.$     Then 
$\frak R$ is called a {\it joint reduction} of $(I_1,\ldots,I_s)$ with respect to $M$ of the type $(k_1,\ldots,k_s)$ if $$I_1^{n_1+1}\cdots I_s^{n_s+1}M = \sum_{i=1}^s(\frak I_i)
I_1^{n_1+1}\cdots I_i^{n_i}\cdots I_s^{n_s+1}M$$ for all large integers $n_1,n_2,\ldots, n_s .$ 
\end{definition}

\begin{lemma}\label{lm36}
Let $M$ be a finitely generated  $A$-module with $\dim M = d > 0$, $J$ an $\frak m$-primary ideal, and $I_1,\ldots, I_s$ ideals of $A.$  Assume that $k_0, k_1,\ldots,k_s$ are non-negative integers such that $k_0+ k_1+\cdots+k_s = d-1$. Let $\frak R = \{x_1, \ldots, x_d\}$ be a joint reduction of $(I_1,\ldots,I_s,J)$ with respect to $M$ of the type $(k_1,\ldots,k_s, k_0+1)$ such that $\frak R$ is a system of parameters for $M$ and $x_1 \in I_1$. Then there exists a non-empty Zariski open subset $U$ of $I_1/\frak mI_1$ such that for any $x \in I_1$ with image $x + \frak mI_1 \in U,$  $\{x, x_2, \ldots, x_d\}$ is a joint reduction of $(I_1,\ldots,I_s,J)$ with respect to $M$ of the type $(k_1,\ldots,k_s, k_0+1)$ and $\{x,x_2, \ldots, x_d\}$ is a system of parameters for $M.$
\end{lemma}

\begin{proof}
By \cite[Lemma 17.1.4]{SH}, a set of elements of $A$ is a joint reduction of $(I_1,\ldots,I_s,J)$ with respect to $M$ if and only if it is so with respect to $A/\mathrm{Ann}_AM.$ Thus we may assume that $M=A=A/\mathrm{Ann}_AM.$ Set $$T=(I_1/\mathfrak{m}I_1)^{k_1} \oplus\cdots\oplus(I_s/\mathfrak{m}I_s)^{k_s}\oplus(J/\mathfrak{m}J)^{k_0+1}.$$ By \cite[Lemma 17.3.2]{SH}, there exists a Zariski open subset $V\subseteq T$ such that $ \{y_1, \ldots, y_d\}$ is a joint reduction of $(I_1,\ldots,I_s,J)$ of the type $(k_1,\ldots,k_s, k_0+1)$ 
 if and only if the image $ (\overline{y}_1, \ldots, \overline{y}_d)$ of $(y_1, \ldots, y_d)$ in $T$ lies in $V.$
 Since $\{x_1, \ldots, x_d\}$ is a joint reduction of $(I_1,\ldots,I_s,J)$ of the type $(k_1,\ldots,k_s, k_0+1)$, $V$ is non-empty. Let $U_1$ be the subset of $I_1/\mathfrak{m}I_1$ consisting of elements $\overline{x}\; (x\in I_1)$ such that $(\overline{x},\overline{x}_2,\ldots,\overline{x}_d)\in V.$ By 
\cite[Lemma 8.5.12]{SH}, $U_1$ is a non-empty Zariski open subset of $I_1/\frak mI_1$. Now, assume  that $P_1,P_2,\ldots, P_t$ are all the prime ideals of $\mathrm{Ass}_A \dfrac{A}{(x_2, \ldots, x_d)}$ such that 
$$\dim_A\dfrac{A}{(x_2, \ldots, x_d)}  = \mathrm{Coht} P_j\quad (1 \le j \le t).$$  For each $j = 1, \ldots, t$, let $W_j$ be the image of $P_j \bigcap I_1$ in $I_1/\frak mI_1$. 
Since $x_1 \in I_1\setminus \bigcup_{j=1}^tP_j$,
$W_1, \ldots, W_t$ are proper $k$-vector subspaces of $I_1/\frak mI_1$ by Nakayama's lemma. Since $k$ is an infinite field,  
$U_2 = (I_1/\frak mI_1) \setminus \bigcup_{j=1}^tW_j$  
is a non-empty Zariski open subset of $I_1/\frak mI_1$. 
Observe that for any $y \in I_1$ with image $y + \frak mI_1 \in U_2,$        $\{y,x_2, \ldots, x_d\}$ is a system of parameters for $M.$ Set $U = U_1 \bigcap U_2.$ Then for any $x \in I_1$ with image $x + \frak mI_1 \in U,$  $\{x, x_2, \ldots, x_d\}$ is a joint reduction of $(I_1,\ldots,I_s,J)$ with respect to $M$ of the type $(k_1,\ldots,k_s, k_0+1)$ and $\{x,x_2, \ldots, x_d\}$ is a system of parameters for $M.$
The lemma has been proved.
\end{proof}
 
The following proposition shows that one can build a joint reduction $\frak R$ as in the state of the main theorem from a Rees's superficial sequence.
\enlargethispage{1cm}

\begin{proposition}\label{pr39}
Let $M$ be a finitely generated  $A$-module with $\dim M = d > 0$ and $I_1,\ldots, I_s$ ideals of $A.$ Set $$h=\max\Big\{\mathrm{ht}\Big(\dfrac{I_j+\mathrm{Ann}_AM} {\mathrm{Ann}_AM}\Big)\ |\ j=1,\ldots,s\Big\}.$$  Then the following statements hold.
\begin{enumerate}[\rm (i)]
\item There exists a Rees's superficial sequence $x_1,\ldots,x_h$ of $(I_1,\ldots, I_s)$ with respect to $M$ which is a part of system of parameters for $M.$
\item If $y_1,\ldots,y_m$ is a Rees's superficial sequence of $(I_1,\ldots, I_s)$ with respect to $M$  
and $I_1\cdots I_s\subseteq\sqrt{(y_1,\ldots,y_m)+\mathrm{Ann}_AM}$ $($in particular, if $y_1,\ldots,y_m$ is a system of parameters for $M )$ then $\{y_1,\ldots,y_m\}$ is a joint reduction of $(I_1,\ldots, I_s)$ with respect to $M$.
\end{enumerate}
\end{proposition}

\begin{proof}
To prove (i), we may assume that $h>0.$ Let $x_1,\ldots,x_l$ be a Rees's superficial sequence of maximal length of $(I_1,\ldots, I_s)$ with respect to $M$ which is a part of system of parameters for $M.$ Set $M'=M/(x_1,\ldots,x_l)M$ and denote by $\Sigma$ the set of minimal prime ideals of $\mathrm{Ann}_AM'$ such that for any $P \in \Sigma,$ $\dim M' = \mathrm{Coht} P.$ \;  
  Since $$\mathrm{ht}\Big(\dfrac{I_i+\mathrm{Ann}_AM'} {\mathrm{Ann}_AM'}\Big)\geq\mathrm{ht}\Big(\dfrac{I_i+\mathrm{Ann}_AM} {\mathrm{Ann}_AM}\Big)-l,$$ $$\begin{aligned}
\max\Big\{\mathrm{ht}\Big(\dfrac{I_i+\mathrm{Ann}_AM'} {\mathrm{Ann}_AM'}\Big)\ |\ i=1,\ldots,s\Big\} &\geq\max\Big\{\mathrm{ht}\Big(\dfrac{I_i+\mathrm{Ann}_AM} {\mathrm{Ann}_AM}\Big)\ |\ i=1,\ldots,s\Big\}-l\\
&=h-l.
\end{aligned}$$
We now need to show that $l = h.$  Indeed, if $0\leq l<h$ then $h-l > 0.$ 
Hence there exists $I_j$ that is contained in no prime ideal belonging to $\Sigma$. By Lemma \ref{lm32}, there is a Rees's superficial element $x_{l+1}\in I_j$ of $(I_1,\ldots, I_s)$ with respect to $M'$ which does not belong to any element of $\Sigma$. It is easily seen that $x_1,\ldots,x_l,x_{l+1}$ is a Rees's superficial sequence of $(I_1,\ldots, I_s)$ with respect to $M$ which is also a part of system of parameters for $M$. This contradicts with  $x_1,\ldots,x_l$ is a sequence of maximal length. Hence $l = h.$ We obtain (i). 

The proof of (ii) is based on the idea in the proof of \cite[Theorem 3.4]{Vi3}. We first show by induction on $m$ that
\begin{equation}\label{eq1}
(y_1,\ldots,y_m)M\bigcap I_1^{n_1+1}\cdots I_s^{n_s+1}M = \sum_{i=1}^s(\frak I_i)
I_1^{n_1+1}\cdots I_i^{n_i}\cdots I_s^{n_s+1}M
\end{equation}
for all large $n_1,\ldots,n_s$, where $\frak{I}_i=\{y_1,\ldots,y_m\}\bigcap I_i,\ i=1,\ldots,s.$ The case $m=1$ is obvious. Assume now that $m>1$, that (\ref{eq1}) has been proved for $m-1$, and that $y_m\in I_s$. Set $\overline{M}=M/(y_1,\ldots,y_{m-1})M.$ Then $y_m$ is a Rees's superficial element of $(I_1,\ldots, I_s)$ with respect to $\overline{M}$, so
$$y_m\overline{M}\bigcap I_1^{n_1+1}\cdots I_s^{n_s+1}\overline{M} =  y_m
I_1^{n_1+1}\cdots I_{s-1}^{n_{s-1}+1}I_s^{n_s}\overline{M}$$
for all large $n_1,\ldots,n_s$. This implies that
$$(y_1,\ldots,y_m)M\bigcap I_1^{n_1+1}\cdots I_s^{n_s+1}M \subseteq y_m I_1^{n_1+1}\cdots I_{s-1}^{n_{s-1}+1}I_s^{n_s}{M}+(y_1,\ldots,y_{m-1})M,$$
and hence
$$\begin{aligned}
&(y_1,\ldots,y_m)M \bigcap I_1^{n_1+1}\cdots I_s^{n_s+1}M \\
\subseteq &(y_m I_1^{n_1+1}\cdots I_{s-1}^{n_{s-1}+1}I_s^{n_s}{M}+(y_1,\ldots,y_{m-1})M)
\bigcap I_1^{n_1+1}\cdots I_s^{n_s+1}M\\
=&\ y_m I_1^{n_1+1}\cdots I_{s-1}^{n_{s-1}+1}I_s^{n_s}{M}
+(y_1,\ldots,y_{m-1})M\bigcap I_1^{n_1+1}\cdots I_s^{n_s+1}M
\end{aligned}$$
for all large $n_1,\ldots,n_s$. By induction hypothesis,
$$(y_1,\ldots,y_{m-1})M\bigcap I_1^{n_1+1}\cdots I_s^{n_s+1}M=\sum_{i=1}^s(\frak I'_i)
I_1^{n_1+1}\cdots I_i^{n_i}\cdots I_s^{n_s+1}M$$
for all large $n_1,\ldots,n_s$, where $\frak{I}'_i=\{y_1,\ldots,y_{m-1}\}\bigcap I_i,\ i=1,\ldots,s.$ Note that $\frak{I}'_i\subseteq\frak{I}_i$ for $i=1,\ldots,s-1$ and $\frak{I}_s=\frak{I}'_s\bigcup\{y_m\}.$ Hence
$$\begin{aligned}
&(y_1,\ldots,y_m)M\bigcap I_1^{n_1+1}\cdots I_s^{n_s+1}M \\
\subseteq &\ y_m I_1^{n_1+1}\cdots I_{s-1}^{n_{s-1}+1}I_s^{n_s}{M}+\sum_{i=1}^s(\frak I'_i)
I_1^{n_1+1}\cdots I_i^{n_i}\cdots I_s^{n_s+1}M\\
\subseteq&\sum_{i=1}^s(\frak I_i)
I_1^{n_1+1}\cdots I_i^{n_i}\cdots I_s^{n_s+1}M
\end{aligned}$$
for all large $n_1,\ldots,n_s$. Since the reverse inclusion is obvious, we get (\ref{eq1}).

Now if $$I_1\cdots I_s\subseteq\sqrt{(y_1,\ldots,y_m)+\mathrm{Ann}_AM}$$ then $$I_1^{n_1+1}\cdots I_s^{n_s+1}M\subseteq (y_1,\ldots,y_m)M$$ for all large $n_1,\ldots,n_s$. It therefore follows from (\ref{eq1}) that $\{y_1,\ldots,y_m\}$ is a joint reduction of $(I_1,\ldots, I_s)$ with respect to $M$.
\end{proof}

As a consequence of Proposition \ref{pr39}, we obtain the following corollary. 

\begin{corollary}\label{co36}
 Let $M$ be a finitely generated  $A$-module with $\dim M = d > 0$, $J$ an $\frak m$-primary ideal, and $I_1,\ldots, I_s$ ideals of $A.$ Set $I= I_1\cdots I_s$. Assume that $$\mathrm{ht}\Big(\dfrac{I+\mathrm{Ann}_AM}{\mathrm{Ann}_AM}\Big) = h >0.$$ Let $k_0, k_1,\ldots,k_s$ be non-negative integers such that  
$$k_0+ k_1+\cdots+k_s = d-1\ \ and\ \ k_1+\cdots+k_s < h.$$
Then there exists a Rees's superficial sequence $x_1,\ldots,x_d$ of $(I_1,\ldots, I_s,J)$ with respect to $M$ of the type $(k_1,\ldots,k_s, k_0+1)$ which is a system of parameters for $M$, and in this case, $\{x_1,\ldots,x_d\}$  
is a joint reduction of $(I_1,\ldots, I_s, J)$ with respect to $M$.
\end{corollary}

\begin{proof} Set $t=k_1+\cdots+k_s.$
We first show that there exists a Rees's superficial sequence $x_1,\ldots,x_t$ of $(I_1,\ldots, I_s,J)$ with respect to $M$ of the type $(k_1,\ldots,k_s, 0)$ which is a part of system of parameters for $M$. Assume that we have built a Rees's superficial sequence of maximal length  $x_1,\ldots,x_l$ of $(I_1,\ldots, I_s,J)$ with respect to $M$ of the type $(l_1,\ldots,l_s, 0)$ which is a part of system of parameters for $M$, where $0\leq l_i\leq k_i$ for $i=1,\ldots,s$ and $l=l_1+\cdots+l_s.$ Set $M'=M/(x_1,\ldots,x_l)M.$ Note that
$$\begin{aligned}
\mathrm{ht}\Big(\dfrac{I_i+\mathrm{Ann}_AM'} {\mathrm{Ann}_AM'}\Big)&\geq\mathrm{ht}\Big(\dfrac{I_i+\mathrm{Ann}_AM} {\mathrm{Ann}_AM}\Big)-l\\
&\geq\mathrm{ht}\Big(\dfrac{I+\mathrm{Ann}_AM} {\mathrm{Ann}_AM}\Big)-l\\
&=h-l \geq h-t>0
\end{aligned}$$
for every $i=1,\ldots,s.$ We need to show that $l_i = k_i$ for $i=1,\ldots, s$ and hence $l = h.$ 
Indeed, if there is $l_j<k_j,$ then with the same argument as in the proof of Proposition \ref{pr39}(i), we can find an element $x_{l+1}\in I_j$ such that $x_1,\ldots,x_l,x_{l+1}$ is a Rees's superficial sequence of $(I_1,\ldots, I_s,J)$ with respect to $M$ of the type $(l_1,\ldots,l_j+1,\ldots,l_s, 0)$ which is also a part of system of parameters for $M$. This contradicts with the maximum of sequence $x_1,\ldots,x_l$. Hence  $l_i = k_i$ for $i=1,\ldots, s.$      
We obtain a sequence $x_1,\ldots,x_t$ as required. Put $\overline{M}=\dfrac{M}{(x_1,\ldots,x_t)M}.$ Since $\mathrm{ht}\Big(\dfrac{J+\mathrm{Ann}_A\overline{M}} {\mathrm{Ann}_A\overline{M}}\Big)=d-t=k_0+1,$ it follows from Proposition \ref{pr39}(i) that there exists a Rees's superficial sequence $x_{t+1},\ldots,x_d\in J$ of $(I_1,\ldots, I_s,J)$ with respect to $ \overline{M}$ of the type $(0,\ldots,0,k_0+1)$ which is a system of parameters for $ \overline{M}$. Thus we get a sequence $x_1,\ldots,x_d$ that is a Rees's superficial sequence of $(I_1,\ldots, I_s,J)$ with respect to ${M}$ of the type $(k_1,\ldots,k_s,k_0+1)$ and this sequence is also a system of parameters for $M$. By
 Proposition \ref{pr39}(ii),
 $\{x_1,\ldots,x_d\}$  
is a joint reduction of $(I_1,\ldots, I_s, J)$ with respect to $M.$ 
\end{proof}
  
\section{Mixed Multiplicities of Ideals}

Let $M$ be a finitely generated  $A$-module of dimension $d>0$. Let $J$ be an $\frak m$-primary ideal, and
$I_1,\ldots,I_s$ ideals such that $I= I_1\cdots I_s$  is not contained in $ \sqrt{\mathrm{Ann}_A{M}}$. 
 Set $\dim M/0_M:I^\infty = q.$  Recall that by  \cite[Proposition 3.1]{Vi} (see also \cite {MV}),    
$$\ell_A\Big(\frac{J^{n_0}I_1^{n_1}\cdots I_s^{n_s}M}{J^{n_0 + 1}I_1^{n_1}\cdots I_s^{n_s}M}\Big)$$
is a polynomial of total degree $q-1$ for all large $n_0,n_1,\ldots, n_s.$     Write the terms of total degree $q-1$ in this polynomial in the form 
$$\sum_{k_0 + k_1 + \cdots + k_s = q - 1} e_A(J^{[k_0+1]},I_1^{[k_1]}, \ldots,I_s^{ [k_s]}, M)\frac{n_0^{k_0}n_1^{k_1}\cdots n_s^{k_s}}{k_0!k_1! \cdots k_s!},$$
then $e_A(J^{[k_0+1]},I_1^{[k_1]}, \ldots,I_s^{ [k_s]}, M)$ are non-negative integers not all zero. One calls  $e_A(J^{[k_0+1]},I_1^{[k_1]}, \ldots,I_s^{ [k_s]}, M)$ the 
{\it mixed multiplicity of $(I_1,\ldots, I_s,J)$ with respect to $M$ of the type $( k_1, \ldots, k_s,k_0+1)$}.
It is easily seen that if $\mathrm{ht}\Big(\dfrac{I+\mathrm{Ann}_AM}{\mathrm{Ann}_AM}\Big)  >0$
then $\dim M/0_M:I^\infty = \dim M = d.$

  In this section, we show that  mixed multiplicities of $(I_1,\ldots,I_s, J)$ with respect to $M$ can be expressed as the multiplicities of ideals generated by joint reductions of them. From this we get the Rees's theorem on mixed multiplicity as a consequence. 
  Moreover,  we also obtain a formula that allows us to compute mixed multiplicities in terms of the multiplicities of ideals generated by Rees's superficial sequences.

The main result of this paper is the following theorem. This was established by Rees for the case of $\frak m$-primary ideals in \cite[Theorem 2.4]{Re}.  
\begin{theorem}\label{th37}
 Let $M$ be a finitely generated $A$-module of Krull dimension $d > 0$. Let $J$ be an $\frak m$-primary ideal, and $I_1,\ldots, I_s$   ideals of $A.$ Set $I= I_1\cdots I_s$. Assume that $\mathrm{ht}\Big(\dfrac{I+\mathrm{Ann}_AM}{\mathrm{Ann}_AM}\Big) = h >0$ and $k_0, k_1,\ldots,k_s$ are non-negative integers such that $k_0+ k_1+\cdots+k_s = d-1$ and $ k_1+\cdots+k_s < h.$ 
Let $\frak R= \{x_1, \ldots, x_d\}$ be a joint reduction of 
 $(I_1,\ldots,I_s,J)$ with respect to $M$ of the type $(k_1,\ldots,k_s, k_0+1)$ such that     
$\frak R$ is a system of parameters for $M.$ Then $$e_A(J^{[k_0+1]}, I_1^{[k_1]},\ldots, I_s^{[k_s]}, M) = e_A(\frak R, M).$$
\end{theorem}
\begin{proof} First, we consider the following fact.

\begin{remark} \label{rm314} If $k_1=\cdots=k_s=0$ then  $\frak R \subset J$ and 
 $\frak R$ is a joint reduction of $(I_1,\ldots,I_s,J)$ with respect to $M$ of the type $(0,\ldots,0, d).$  Hence there exists an integer $u$ such that $J^{n+1}I^uM=J^{n+1}I_1^u\cdots I_s^uM= (\frak R)J^{n}I_1^u\cdots I_s^uM= (\frak R)J^{n}I^uM$ for all $n\geqslant u.$ This means that $(\frak R)$ is a reduction of $J$ with respect to $I^uM$ by \cite{NR}. So by \cite[Theorem 1]{NR} we have 
\begin{equation}\label{eq2}
e_A(J,I^uM)=e_A((\frak R),I^uM). 
\end{equation}
Since $\mathrm{ht}\Big(\dfrac{I+\mathrm{Ann}_AM}{\mathrm{Ann}_AM}\Big) >0,$ it follows that $\dim (M/I^uM)<\dim M.$  Hence  
\begin{equation}\label{eq3}
e_A(J, M)=e_A(J,I^uM) \;\;\mathrm{ and }\;\; e_A((\frak R),M)=e_A((\frak R),I^uM).
\end{equation}  
By (\ref{eq2}) and (\ref{eq3}) we obtain   
$e_A(J, M)= e_A((\frak R),M).$ Remember that $\frak R$ is a system of parameters for $M,$     
$e_A((\frak R), M)=e_A(\frak R, M).$ 
Consequently $e_A(J, M)= e_A(\frak R,M).$ Since $\mathrm{ht}\Big(\dfrac{I+\mathrm{Ann}_AM}{\mathrm{Ann}_AM}\Big) >0,$
by \cite[Lemma 3.2(ii)]{Vi} (see also \cite[Lemma 3.2(ii)]{MV}) we get
$e_A(J^{[d]}, I_1^{[0]},\ldots, I_s^{[0]},M)=e_A(J,M).$
Therefore,
$e_A(J^{[d]}, I_1^{[0]},\ldots, I_s^{[0]},M)=e_A(\frak R,M).$
\end{remark}

 We now prove the theorem by induction on $d$ that if 
$\frak R$ is a joint reduction of 
 $(I_1,\ldots,I_s,J)$ with respect to $M$ of the type $(k_1,\ldots,k_s, k_0+1)$ such that     
$\frak R$ is a system of parameters for $M$ then $e_A(J^{[k_0+1]}, I_1^{[k_1]},\ldots, I_s^{[k_s]}, M) = e_A(\frak R, M).$
   If $d = 1$ then $k_1=\cdots=k_s=0.$ It follows from  Remark \ref{rm314} that $e_A(J^{[1]}, I_1^{[0]},\ldots, I_s^{[0]},M)=e_A(\frak R,M).$

Next, consider the case $d>1$. Since $\frak R$ is a joint reduction of $(I_1,\ldots,I_s,J)$ with respect to $M$ of the type $(k_1,\ldots,k_s, k_0+1)$, it is also a joint reduction of $(I_1,\ldots,I_s,J)$ with respect to $A/\mathfrak{p}$ of the type $(k_1,\ldots,k_s, k_0+1)$, where $\mathfrak{p}$ is a minimal prime ideals of $\mathrm{Ann}_AM$ (see e.g \cite[Lemma 17.1.4]{SH}). Denote by $\Lambda$ the set of minimal prime ideals $\mathfrak{p}$ of $\mathrm{Ann}_AM$ such that $\dim A/\mathfrak{p}=d.$ For any $\mathfrak{p} \in \Lambda,$ set $B = A/\mathfrak{p},$ then $B$ is an $A$-module with $\dim B = d$ and $\frak R$ is a system of parameters for $B.$ We show that the theorem is true for $B$, i.e.,
\begin{equation}\label{eq4}
 e_A(J^{[k_0+1]}, I_1^{[k_1]},\ldots, I_s^{[k_s]}, B) = e_A(\frak R, B).
\end{equation}
We need the following  comment. 

\begin{remark} \label{rm34}
If $\mathfrak{p}\in\Lambda$ then $\mathrm{ht}\Big(\dfrac{I+\mathfrak{p}}{\mathfrak{p}}\Big) \geqslant h$. Indeed, without loss of generality, we may assume that $\mathrm{Ann}_AM=0.$
 It is enough to show that there are $h$ elements $a_1,\ldots,a_h$ in $I$ such that 
\begin{equation}\label{eq5}
\mathrm{ht}(a_1,\ldots,a_h) =\mathrm{ht}\Big(\dfrac{(a_1,\ldots,a_h)+\mathfrak{p}} {\mathfrak{p}}\Big) =h.
\end{equation}
If $h=1$, this is trivial since $I\nsubseteq\mathfrak{p}$. Assume that $h>1$ and we have found $a_1,\ldots,a_{h-1}\in I$ satisfying $\mathrm{ht}(a_1,\ldots,a_{h-1}) =\mathrm{ht}\Big(\dfrac{(a_1,\ldots,a_{h-1})+\mathfrak{p}} {\mathfrak{p}}\Big) =h-1.$ We claim that $\mathrm{ht}((a_1,\ldots,a_{h-1})+\mathfrak{p})=h-1$. Indeed, if $\mathrm{ht}((a_1,\ldots,a_{h-1})+\mathfrak{p})>h-1$ then we can choose $b\in\mathfrak{p}$ such that $\mathrm{ht}(a_1,\ldots,a_{h-1},b) = h$, and hence there exist $a_{h+1},\ldots,a_d$ in $A$ such that $\mathfrak{q}=(a_1,\ldots,a_{h-1},b,a_{h+1},\ldots,a_d)$ is an $\mathfrak{m}$-primary ideal. In this case, the image of $\mathfrak{q}$ in $A/\mathfrak{p}$ is primary to the maximal ideal $\mathfrak{m}/\mathfrak{p}$. But this contradicts with $\dim A/\mathfrak{p}=d$ since
$$\mathrm{ht}\Big(\dfrac{\mathfrak{q}+\mathfrak{p}} {\mathfrak{p}}\Big) =\mathrm{ht}\Big(\dfrac{(a_1,\ldots,a_{h-1},a_{h+1},\ldots,a_d)+\mathfrak{p}} {\mathfrak{p}}\Big)\leqslant d-1.$$
So $\mathrm{ht}((a_1,\ldots,a_{h-1})+\mathfrak{p})=h-1$. Thus we can choose $a_h\in I$ that avoids all the minimal prime ideals of $(a_1,\ldots,a_{h-1})$ and of $(a_1,\ldots,a_{h-1})+\mathfrak{p}.$ Observe that the elements $a_1,\ldots,a_h$ satisfy (\ref{eq5}).
\end{remark}

Return to the proof of the theorem. Since $\mathrm{Ann}_AB = \mathfrak{p},$ 
$$ k_1+\cdots+k_s < h \le \mathrm{ht}\Big(\dfrac{I+\mathrm{Ann}_AB}{\mathrm{Ann}_AB}\Big)$$
by Remark \ref{rm34}. Consequently, $B$ satisfies the assumptions of the theorem.
By \cite[Proposition 3.1(vi)]{Vi2} (see also \cite[Theorem 3.6(ii)]{MV}), $e_A(J^{[k_0+1]},             I_1^{[k_1]},\ldots, I_s^{[k_s]}, B) \ne 0.$ If $k_1=\cdots=k_s=0$ then (\ref{eq4}) is true by Remark \ref{rm314}. If    
 $k_1,\ldots,k_s$ are not all zero we may assume that $k_1>0$ and $x_1\in I_1$. By Proposition \ref{lm33}, there is a non-empty Zariski open subset $U_1$ of $I_1/\frak{m}I_1$ such that if $y\in I_1$ with image $y+\frak{m}I_1\in U_1$ then $y$ is a weak-(FC)-element of $(I_1,\ldots,I_s,J)$ with respect to $B$. Moreover, by Lemma \ref{lm36}, there also exists a non-empty Zariski open subset $U_2$ of $I_1/\frak{m}I_1$ such that whenever $z\in I_1$ with image $z+\frak{m}I_1\in U_2,$ then $\{z,x_2,\ldots,x_d\}$ is a joint reduction of $(I_1,\ldots,I_s,J)$ with respect to $B$ and $\{z,x_2,\ldots,x_d\}$ is  a system of parameters for $B.$ Now choose $x\in I_1$ such that $x +\frak{m}I_1\in U_1\bigcap U_2$ then $x$ is a weak-(FC)-element of $(I_1,\ldots,I_s,J)$ with respect to $B$ and $\{x,x_2,\ldots,x_d\}$ is a joint reduction of $(I_1,\ldots,I_s,J)$ with respect to $B$ of the type $(k_1,\ldots,k_s, k_0+1)$ and $\{x,x_2,\ldots,x_d\}$ is also  a system of parameters for $B.$ Set $\overline{B}=B/xB$.  Since $x$ is a weak-(FC)-element and $e_A(J^{[k_0+1]}, I_1^{[k_1]},\ldots, I_s^{[k_s]}, B) \ne 0$, $x$ is an (FC)-element by
\cite[Proposition\ 3.1(i)]{Vi2}. Hence 
 by \cite[Proposition\ 3.3]{Vi} (see also \cite{DV} and  \cite[Proposition\ 3.3]{MV})
 we have $$e_A(J^{[k_0+1]}, I_1^{[k_1]},\ldots, I_s^{[k_s]}, B)=e_{A}(J^{[k_0+1]}, I_1^{[k_1-1]},\ldots, I_s^{[k_s]}, \overline{B}).$$ Note that $\{x,x_2,\ldots,x_d\}$ is a joint reduction of $(I_1,\ldots,I_s,J)$ with respect to $B$ of the type $(k_1,\ldots,k_s, k_0+1),$ $\{x_2,\ldots,x_d\}$ is a joint reduction of $(I_1,\ldots,I_s,J)$ with respect to $\overline{B}$ of the type $(k_1-1,\ldots,k_s, k_0+1)$. Moreover, since $\{x,x_2,\ldots,x_d\}$ is a system of parameters for $B,$ $\dim \overline{B} = d-1$ and $\{x_2,\ldots,x_d\}$ is a system of parameters for $\overline{B}.$ Note also that $\mathrm{ht}\Big(\dfrac{I+\mathrm{Ann}_A\overline{B}}{\mathrm{Ann}_A\overline{B}}\Big)\geq \mathrm{ht}\Big(\dfrac{I+\mathrm{Ann}_AB}{\mathrm{Ann}_AB}\Big) -1>(k_1-1)+k_2+\cdots+k_s$. Hence by induction hypothesis,
$e_{A}(J^{[k_0+1]}, I_1^{[k_1-1]},\ldots, I_s^{[k_s]}, \overline{B})
= e_{A}(x_2,\ldots,x_d,\overline{B})
$. Since $x\not\in\mathfrak{p},$  $x$ is not  a zero divisor on $B.$ Therefore,   
$$ e_{A}(x_2,\ldots,x_d,\overline{B}) = e_{A}(x, x_2,\ldots,x_d,B)$$ by \cite{AB}.
Thus, $e_{A}(J^{[k_0+1]}, I_1^{[k_1]},\ldots, I_s^{[k_s]}, B)
= e_{A}(x, x_2,\ldots,x_d,B).$
We now show that $$e_{A}(x,x_2,\ldots,x_d,B)=e_{A}(x_1,x_2,\ldots,x_d,B).$$ Put $t=\max\{k_0,k_1-1,k_2,\ldots,k_s\}.$ 
 If $t>0$ we may assume that $k_i>0$ (or $k_i>1$ if $i=1$) and $x_d\in I_i.$ Set $B'=B/x_dB$. Then $\{x,x_2,\ldots,x_{d-1}\}$ and $\{x_1,x_2,\ldots,x_{d-1}\}$ are both systems of parameters for $B'.$ Moreover, they are both joint reductions of $(I_1,\ldots,I_s,J)$ with respect to $B'$ of the type $(k_1,\ldots,k_i-1,\ldots,k_s, k_0+1)$. Similar as above, we also have $\mathrm{ht}\Big(\dfrac{I+\mathrm{Ann}_AB'}{\mathrm{Ann}_AB'}\Big) >k_1+\cdots+k_i-1+\cdots+k_s.$
 So by induction hypothesis,
$$\begin{aligned}
e_{A}(x,x_2,\ldots,x_{d-1},B')&=e_{A}(J^{[k_0+1]}, I_1^{[k_1]},\ldots, I_i^{[k_i-1]},\ldots, I_s^{[k_s]}, B')\\
&=e_{A}(x_1,x_2,\ldots,x_{d-1},B').
\end{aligned}$$
Since $x_d$ is also not  a zero divisor of $B$, by \cite{AB} we have
$$\begin{aligned}
e_{A}(x,x_2,\ldots,x_{d-1},B')&=e_{A}(x,x_2,\ldots,x_d,B),\\
e_{A}(x_1,x_2,\ldots,x_{d-1},B')&=e_{A}(x_1,x_2,\ldots,x_d,B).
\end{aligned}$$
Therefore, $e_{A}(x,x_2,\ldots,x_d,B)=e_{A}(x_1,x_2,\ldots,x_d,B).$ 

If $t=0$ then $k_0=k_2=\cdots=k_s=0,\ k_1=1$, and $d=2.$ In this case, $x_2\in J$ and $\dim B/x_2B = 1.$  Since $\mathrm{ht}\Big(\dfrac{I+\mathfrak{p}}{\mathfrak{p}}\Big) \geqslant h>k_1
+\cdots+k_s=1,$ $I_1,\ldots,I_s$ are ideals of definition of $B$. It follows that $I_1,\ldots,I_s$ are ideals of definition of $B/x_2B$. As $\dim B/x_2B = 1$, applying Remark \ref{rm314}  for modules of dimension 1 with $I_1$ playing the role of $J,$     
 we get
$$e_{A}(x,B/x_2B)=e_A(J^{[0]}, I_1^{[1]}, I_2^{[0]},\ldots, I_s^{[0]}, B/x_2B)=e_{A}(x_1,B/x_2B),$$
 and hence $e_{A}(x,x_2,B)=e_{A}(x_1,x_2,B).$ By the above  results,  we obtain  $$e_{A}(x,x_2,\ldots,x_d,B)=e_{A}(x_1,x_2,\ldots,x_d,B).$$ 
Hence  
$e_{A}(J^{[k_0+1]}, I_1^{[k_1]},\ldots, I_s^{[k_s]}, B) =
e_{A}(x_1,x_2,\ldots,x_d,B).$ This proves (\ref{eq4}).
Thus for any $\mathfrak{p} \in \Lambda,$ we have $e_{A}(J^{[k_0+1]}, I_1^{[k_1]},\ldots, I_s^{[k_s]}, A/\mathfrak{p}) =
e_{A}(x_1,x_2,\ldots,x_d,A/\mathfrak{p}).$ Consequently,    
$$\sum_{\mathfrak{p}\in \Lambda}\ell_A(M_{\mathfrak{p}})e_{A}(J^{[k_0+1]}, I_1^{[k_1]},\ldots, I_s^{[k_s]}, A/\mathfrak{p}) = \sum_{\mathfrak{p}\in \Lambda}\ell_A(M_{\mathfrak{p}})e_{A}(x_1,x_2,\ldots,x_d, A/\mathfrak{p}).$$ Remember that $e_A(x_1,\ldots,x_d, M) =\sum_{\mathfrak{p}\in\Lambda}\ell_A(M_{\mathfrak{p}})e_A(x_1,\ldots,x_d, A/\mathfrak{p})$ (see \cite[Theorem 11.2.4]{SH}). 
Now by \cite[Corollary 3.6]{VT1}, $$e_A(J^{[k_0+1]}, I_1^{[k_1]},\ldots, I_s^{[k_s]}, M) =\sum_{\mathfrak{p}\in \Lambda}\ell_A(M_{\mathfrak{p}})e_{A}(J^{[k_0+1]}, I_1^{[k_1]},\ldots, I_s^{[k_s]}, A/\mathfrak{p}).$$ 
Therefore, $e_A(J^{[k_0+1]}, I_1^{[k_1]},\ldots, I_s^{[k_s]}, M)= e_A(x_1,\ldots,x_d, M).$
The proof is complete.
\end{proof}

\begin{remark}\label{rm20} Keep the notation of Theorem \ref{th37}. Now if $x_1, \ldots, x_d$ is   
    a  Rees's superficial sequence of 
 $(I_1,\ldots,I_s,J)$ with respect to $M$ of the type $(k_1,\ldots,k_s, k_0+1)$ that     
 is a system of parameters for $M$ then  $\{x_1, \ldots, x_d\}$ is a joint reduction  of 
 $(I_1,\ldots,I_s,J)$ with respect to $M$ of the type $(k_1,\ldots,k_s, k_0+1)$ by Corollary \ref {co36}. 
Hence $e_A(J^{[k_0+1]}, I_1^{[k_1]},\ldots, I_s^{[k_s]}, M)= e_A(x_1,\ldots,x_d, M)$ by Theorem \ref{th37}.

\end{remark}

\begin{remark}\label{rm35}
From Theorem \ref{th37}, one may raise a question:
 Does the theorem hold if $ k_1+\cdots+k_s \geqslant h$ ?
 Consider the case $s=1$ and $M=A.$ Let $I_1=I$ be an equimultiple ideal of $A$, i.e., an ideal such that $\mathrm{ht}I=s(I)$, where $s(I)=\dim \bigoplus_{n\geqslant0}I^n/\frak{m}I^n.$ If $\mathrm{ht}I=h>0$ then by \cite[Theorem 3.1(iii)]{Vi4}, $e_A(J^{[d-i]}, I^{[i]})=0$ for all $i\geqslant h.$ Therefore, when $i\geqslant h$, $e_A(J^{[d-i]}, I^{[i]})$ can not be multiplicity of any system of parameters. This example shows that Theorem \ref{th37} does not hold in general if one omits  the assumption $ k_1+\cdots+k_s < h$. 
\end{remark}

Finally, we recover Rees's theorem, which is a motivation for this paper.
 \begin{corollary}[$\mathrm{\cite[Theorem\ 2.4(i),(ii)]{Re}}$]\label{co34}
 Let $M$ be a finitely generated  $A$-module with $\dim M = d > 0$. Let $I_1,\ldots, I_s$ be $\frak m$-primary ideals of $A$. Assume that $k_1,\ldots,k_s$ are non-negative integers such that $k_1+\cdots+k_s = d$  and $\frak R$ is a joint reduction of $(I_1,\ldots,I_s)$ with respect to $M$ of the type $(k_1,\ldots,k_s)$.
Then $e_A(I_1^{[k_1]},\ldots, I_s^{[k_s]}, M) = e_A(\frak R, M).$
\end{corollary}

\begin{proof}
 If $s=1$ then $(\frak R)$ is a reduction of $I_1$ with respect to $M,$ i.e., $(\frak R)I_1^nM = I_1^{n+1}M$ for all $n \gg 0$ \cite {NR}.  So by \cite[Theorem 1]{NR}, we have 
$e_A(\frak R, M)=e_A(I_1,M)=e_A(I_1^{[d]},M).$ Assume now that $s>1$. Since $I_1,\ldots, I_s$ are $\frak m$-primary ideals,
  $\frak R$ is a system of parameters for $M.$ Moreover, from $k_1+\cdots+k_s = d>0$, we may assume $k_1>0$. Then $k_2+\cdots+k_s<d$. So by Theorem \ref{th37}, we get the proof of this corollary.    
\end{proof}

\end{document}